\documentclass[11pt,twoside]{article}
\usepackage{amssymb}
\usepackage[mathscr]{eucal}
\usepackage{eufrak}
\usepackage{amsmath,amsthm}
\usepackage{mathrsfs}
\usepackage[colorlinks=true,
   urlcolor=blue,           filecolor=green,      
   citecolor=green,      
   linkcolor=red,           bookmarks=true,
  unicode,
   plainpages=false,   ]{hyperref}
\newcommand{\half}{\mbox{$\frac12$}}
\newcommand{\essup}[1]{{\rm ess}\,{{\displaystyle \sup_{\hspace*{-5mm}{#1}}}}}
\newcommand{\real}{{\mathbb R}}
\usepackage{color}

\def\bp{\begin{proof}}
\def\ep{\end{proof}}
\def\n{\nabla}

\def\sfrac#1#2{\mbox{\Large$\frac{#1}{#2}$}}
\def\intl#1{\int\limits_{#1}}
\def\intll#1#2{\int\limits_{#1}^{#2}}
\def\dm{|\hskip-0.05cm|}

\def\OO{\Omega}
\def\displ{\displaystyle}

\def\VS{\vspace{6pt}\\\displ }
\def\rf#1{{\rm(\ref{#1})}}
\def\R{\Bbb R}

\def\be{\begin{equation}}
\def\ba{\begin{array}}
\def\ea{\end{array}}
\def\ee{\end{equation}}
\def\ov{\overline}
\def\po{{\partial\Omega}}
\newtheorem{lemma}
{\bf Lemma} 
\font\sc=cmcsc10
\begin{document}
\newtheorem{ass}
{\bf Assumption} 
\newtheorem{defi}
{\bf Definition} 
\newtheorem{tho}
{\bf Theorem} 
\newtheorem{rem}
{\sc Remark} 
\newtheorem{coro}
{\bf Corollary} 
\newtheorem{prop}
{\bf Proposition}
\renewcommand{\theequation}{\arabic{section}.\arabic{equation}}
\title{On the Asymptotic Behavior in Time of the Kinetic Energy \\
in a Rigid Body--Liquid  Problem\\ } 
\medskip\bigskip
\author{Giovanni P. Galdi \& Paolo Maremonti}
\date{\em To Vsevolod Alekseevich Solonnikov, in memoriam}
\maketitle
\setcounter{section}{1}
\begin{abstract} We give sufficient conditions on the initial data for the decay in time of the kinetic energy, $E$, of solutions to the system of equations describing the motion of a rigid body in a Navier-Stokes liquid. More precisely, assuming the initial data  ``small" in appropriate norm, we show that if, in addition, the initial velocity field of the liquid, $v_0$, is  in $L^q$, $q\in(1,2)$, then $E(t)$ vanishes as $t\to\infty$ with a specific order of decay. The order remains, however, unspecified if $v_0\in L^2$.   
\end{abstract}
\section*{{Introduction}}
Let  $\mathscr B$ be a rigid body (closure of a simply connected bounded domain of $\real^3$) moving  in an otherwise quiescent  Navier-Stokes liquid $\mathscr L$ that fills the entire space outside $\mathscr B$.  The motion of the coupled system $\mathscr S:=\mathscr B\cup\mathscr L$ is generated  by prescribed initial data.    
\par 
One of the interesting questions that has recently attracted the attention of several mathematicians, is the behavior  of $\mathscr S$ for large times. More specifically, in \cite{ET} for $\mathscr B$ a sphere, and in \cite{GaLTB,Tu} in the general case, under diverse assumptions on the initial data and driving mechanism, it is shown  that as time grows indefinitely large, the velocity  $v$ of $\mathscr L$ as well as  translational ($\xi$) and angular ($\omega$) velocities of $\mathscr B$ will tend to 0 in certain norms; see also \cite{Hi}. As a matter of fact, in \cite{ET,Tu} it is also shown the non-trivial result that the center of mass $G$  of $\mathscr B$ (in absence of external forces and torques) will necessarily cover a finite distance.  
\par
Notwithstanding the above efforts, in the  general and more interesting case of a body of {\em arbitrary} shape, the following, physically significant question is left open. Let
$$
E(t):=\half\left\{\rho\int_\Omega|v(t)|^2+m|\xi(t)|^2+\omega(t)\cdot {\sf I}\cdot\omega(t)\right\}
$$
with $m$ and ${\sf I}$ mass and inertia tensor of $\mathscr B$, 
denote the total kinetic energy of $\mathscr S$. Is it true that, as $t\to\infty$, $E(t)$ will eventually vanish? The answer to this question might by no means be trivial since, even for the classical Navier-Stokes problem, it was (positively) answered \cite{Ma,Sc} only several decades after its original formulation due to Leray  \cite{Le}.  
In this respect, in \cite{GaLTB}, under the assumption that the initial data have ``small" kinetic energy and dissipation,  it is proved that, as $t\to\infty$, while $\xi(t)$ and $\omega(t)$ tend pointwise to 0,  the velocity field $v$ tends to 0 in the $L^q$-norm only for some $q>2$, thus excluding the case $q=2$, representative of the kinetic energy of $\mathscr L$. To be precise, denoting by $\|\cdot\|_q$ the $L^q$-norm, in \cite{GaLTB} it is proved $\lim_{t\to\infty}\|v(t)\|_6=0$. However, by elementary interpolation, for any $q\in (2,6)$, we have $\|v(t)\|_q\le \|v(t)\|_2\|v(t)\|_6$, which, since $\|v(t)\|_2$ is uniformly bounded \cite[Theorem 2.1]{GaLTB},  shows the claimed property. Moreover, in \cite{Tu}, under further assumptions on the initial data, it is also provided a rate of decay for $\xi$ and $\omega$ and, for $v$,  in $L^q$-norms, but, again, only for $q>2$.   
\par 
More recently, the above question has been addressed in \cite{GaK}, even under the more general case of external forces acting on $\mathscr B$. There, under the hypothesis that the initial data, $v_0,\omega_0$ in addition to having ``small" finite total kinetic energy and dissipation, are also such that $v_0$ is $L^2$-summable with an increasing logarithmic weight, it is shown that $E(t)$ does tend to 0, provided $G$ is kept fixed, namely, $\xi(t)\equiv 0$. In \cite{GaK} it is also noticed that the method there used does not work if $\xi(t)\not\equiv0$.   
\par   
In this article, we address the problem of the time decay of the kinetic energy, with a two-fold objective. Firstly, we remove the assumption $\xi(t)\equiv 0$ made in \cite{GaK}, and in addition, we  also provide a rate of decay for $E(t)$. Precisely, we prove the latter if the initial data satisfy the conditions stated in \cite{GaLTB} with the further hypothesis that $v_0$
 is in $L^q$, for some $q\in [1,2)$; see Theorem \ref{LDD}. If, however, $v_0$ is in $L^2$, then we can only say that $E(t)\to 0$, with no specific order of decay; see Theorem \ref{ABLTN}. Our findings  are sharp, in the sense that in the case $\xi\equiv\omega\equiv 0$, they furnish exactly the same ones for classical Navier-Stokes problem, which are known to be  optimal. Concerning the method we use, it consists of a suitable combination of the methods employed  in \cite{Ma} and \cite{GaLTB}.   
\par
The plan of the article is as follows. In the next section, we recall the relevant equations along with some side results from \cite{GaLTB}.  Moreover,  we prove a key result regarding a time-growth estimate of the $L^q$-norm of $v$; see Proposition \ref{LQ_1}. Finally, in Section 3, we give a proof of our main result, presented in Theorems \ref{LDD} and \ref{ABLTN}.  
\section{Governing Equations and Preliminary Results} 
As somehow customary in this type of study, we shall describe te motion of $\mathscr B$ from a body-fixed frame. In this way,  the domain, $\Omega$, occupied by the liquid becomes time-independent. We suppose $\Omega$ of class $C^2$. Thus, also taking into account the notation introduced in the previous section, the relevant equations describing the motion of the coupled system body-liquid in absence of external forces read as follows \cite{GaRev}
\be\label{NS}\ba{l}\left.\ba{l}\medskip v_t=\n\cdot \mathbb T(v,\pi)-[(v-V)\cdot\n v+\omega\times v]\\
\n\cdot v=0
\ea\right\}
\mbox{\, in } (0,\infty)\times\OO\,,\VS v=V\mbox{\, on }(0,\infty)\times \po\,,\;\lim_{|x|\to\infty}v(t,x)=0\,,t\in(0,T)\,,\\\;m\hskip0.025cm\dot\xi+m\hskip0.025cm\omega\times\xi+\!\!\displ\intl{\po}\hskip-0.1cm\mathbb T(v,\pi)\hskip-0.05cm\cdot\hskip-0.05cm n=0\,,\\ {\sf I}\cdot\dot\omega+\omega\times{\sf I}\hskip-0.05cm\cdot\hskip-0.05cm\dot\omega+\!\!\displ\intl{\po}\hskip-0.1cmx\times \mathbb T(v,\pi)\hskip-0.05cm\cdot\hskip-0.05cm n=0\,,\,t\!\in\!(0,T)\,,\VS
 v(0,x)=v_0(x)\,\mbox{ on }\{0\}\times\OO\,,\; \xi(0)=\xi_0\,,\;\omega(0)=\omega_0\,.\ea\ee
In these equations, $\pi$ is the pressure field of ${\cal L}$,  and $V(x,t):= \xi(t)+ \omega (t) \times x$.  Also, $\mathbb T$ is the Cauchy stress tensor given by
$$
\mathbb T(v,\pi)=2\mathbb D(u)-\pi\,\mathbb I,\ \ 2\mathbb D(v):=\nabla v+(\nabla v)^\top\,,
$$
with  $\mathbb I$ identity tensor. Since our final result is independent of their actual values we set, for simplicity, both density and shear viscosity of $\mathscr L$  equal to 1. 
Moreover,  $m$  denotes the mass of  $\mathscr B$ and ${\sf I}$ its inertia tensor relative to $G$.
\par
We now begin to recall several preparatory results proved in \cite{GaLTB} and some of their corollaries.
Let~\footnote{We shall use standard notation for function
spaces, see \cite{ad}. So, for instance, 
$L^q(\mathcal A)$, $W^{m,q}(\mathcal A)$, $W_0^{m,q}(\mathcal A)$,
etc., will denote the usual Lebesgue and Sobolev spaces on the
domain $\mathcal A$, with norms
$\|\,\cdot\|_{q,\mathcal A}$ and $\|\,\cdot\|_{m,q,\mathcal A}$, respectively.
Whenever confusion will not arise, we shall omit the subscript $\mathcal A$.
Occasionally, for $X$ a Banach space, we denote by $\|\cdot\|_X$ its associated norm. Moreover $L^q(I;X)$, $C(I;X)$ $I$  real interval,  denote classical Bochner spaces.}
$$
{\mathcal R}:=\{\overline{v}\in C^{\infty}(\mathbb R^3): \overline{v}(x)= \overline{v}_1+\overline{v}_2 \times x, \,\ \overline{v}_i \in \mathbb R^3\,, \ i=1,2\}\,,
$$
and define
$$
{\cal V}(\Omega) =   \{ v \in W^{1,2}(\Omega): \nabla\cdot v =0 \textrm{ in }\Omega, \textrm{ }\textrm{ } v\left|_{\Sigma}\right.= \overline{v}, \ \mbox{for some} \  \overline{v}   \in {\cal R}\}.
$$
With the origin in the interior of $\mathscr B$,  shall use the notation
$$
B_R:\{x\in\real^3:|x|<R\}\,,\ \Omega_R:=\Omega\cap B_R,,\ \Omega^R:=\Omega\backslash\overline{B_R}\,,\ R>R_*={\rm diam}\,\mathscr B\,.
$$
The following lemma ensures global existence of solutions to (\ref{NS}) in a suitable function class, and is a particular case of \cite[Theorem 2.1]{GaLTB}.
\begin{lemma}\label{GESD}{\sl There is a $\delta>0$ such that for all $v_0\in \mathcal V(\OO)$, with $v_0=V_0:=\xi_0+\omega_0\times x$ on $\po$, satisfying 
\be\label{GESD-I}\dm v_0\dm_{1,2}+|\xi_0|+|\omega_0|\leq \delta\,,
\ee 
there exists  a corresponding solution $(v,\pi,\xi(t),\omega(t))$  to \rf{NS} defined on $(0,\infty)\times\OO $ such that
\be\label{GESD-II}\ba{l} v\in L^\infty(0,\infty;W^{1,2}(\OO))\mbox{ with }\n v\in L^2(0,\infty;W^{1,2}(\OO))\,, \;\n\pi\in L^2(0,\infty;L^2(\OO))\,, \VS v\in C([0,T);W^{1,2}(\OO_R))\,,\;v_t\,,\;\pi\in L^2(0,\infty;L^2(\OO_R))\,\;\mbox{ for all } T>0\ \mbox{and}\ R\geq R_*\,,\VS\xi,\,\omega\in W^{1,2}(0,\infty)\cap C([0,\infty))\,.\ea\ee
In particular, the following estimate hold 
\be\label{goes}
\essup{t\ge 0}\left(\|v(t)\|_{1,2}+|\xi(t)|+|\omega(t)|\right)+\int_0^\infty\left(|\xi(t)|^2+|{\omega}(t)|^2+|\dot\xi(t)|^2+|\dot{\omega}(t)|^2+\|\n v(t)\|^2_{1,2}\right){\rm d} t\le \kappa\,\delta\,,
\ee
where $\kappa=\kappa(\Omega)>0$. 
}\end{lemma}
The next lemma is shown in \cite[Lemma 3.2]{GaLTB}
\begin{lemma}\label{LGPG}{\sl Let $(v,\pi,\xi,\omega)$ be the solution given in Lemma\,\ref{GESD}. Then, for all $q_1\in(1,6]$ and $q_2\in(\frac32,\infty]$ we have \be\label{LGPG-I}\n\pi\in L^{q_1}(\Omega^{2R_*})\,,\;\ \pi\in L^{q_2}(\OO^{2R_*})\,,\mbox{ a.e. in }t>0\,.\ee In particular, 
\be\label{LGPG-II}\dm \n\pi\dm_{q_1}\leq c(\dm v\cdot \n v\dm_{q_1}+\dm \n\pi\dm_2)\,,\mbox{ a.e. in }t>0\,.\ee}\end{lemma}
\par
{\em Throughout this paper we denote by $C$ a positive constant depending only on $\Omega$ and the norms of the initial data entering equation \rf{GESD-I}}.
\par
The following result is given in \cite[Proposition 2.1]{GaK} in the case $\xi\equiv0$. However, the proof remains basically the same also if $\xi\not\equiv 0$ and, therefore, it will be omitted.
\begin{lemma}\label{GPSMD}{Let $(v,\pi,\xi,\omega)$ be the solution given in Lemma\,\ref{GESD}.
Then,   \be\label{GPSMD-I}\intll0\infty\dm \pi(t)\dm_2^2{\rm d}\,t\leq C\,.\ee}
\end{lemma}
We also have
\begin{lemma}\label{ASS}{Let $(v,\pi,\xi,\omega)$ be the solution given in Lemma\,\ref{GESD}. Then
\be\label{ASS-I}\displ\intll0\infty\dm v(t)\dm_\infty^2{\rm d}\,t+\displ \intll0\infty\|\mathbb T(v,\pi)\cdot n\dm_1^2{\rm d}t\leq C\,.\ee}
\end{lemma}
\bp The proof is a consequence of the following classical embedding and trace inequalities
$$\|v\|_\infty\le c\,(\|v\|_6+\|D^2v\|_2)\,;
\ \ 
\|\mathbb T(v,\pi)\cdot n\dm_1\le c\,(\|\nabla v\|_{1,2}+\|\pi\|_{1,2})\,,
$$
combined with Lemmas \ref{GESD} and \ref{GPSMD}.
\ep
With all the above results in hand, we are now able to prove the following estimate that is of the essence to the proof of our main result.
\begin{prop}\label{LQ_1}{\sl   Let $(v_0,\xi_0,\omega_0)\in(\mathcal V(\OO)\cap L^q(\OO))\times \R\times\R$,  $q\in(1,2)$, satisfy (\ref{GESD-I}). Then, the corresponding solution  given in Lemma\,\ref{GESD} obeys the following estimate
\be\label{LQ-I}\dm v(t)\dm_q\leq \dm v_0\dm_q+C\,t^{\frac{2-q}{2q}}\,,\mbox{ for all }t>0\,.\ee   }\end{prop}
\bp Let  $h_R=h_R(|x|)\in[0,1]$ be a smooth, radial cut-off function with $h_R=1$ for $|x|<\frac R2$, $h_R=0$ for $|x|>2R$ and $R|\n h_R|+R^2|\n\n h_R|\leq c$.
We then multiply both sides of the equation \rf{NS}$_1$ by $h_Rv(|v|^2+\eta)^\frac{q-2}2$, where $\eta>0$. Integrating by parts on $(0,t)\times \OO$, we obtain 
\be\label{LQ-O}\ba{l}\displ\sfrac d{dt}\dm (|v(t)|^2+\eta)^\frac 12h_R\dm_q^q+\dm h_R\n v(|v|^2+\eta)^\frac{q-2}4\dm_2^2+(q-2)\dm (\n v\cdot v)h_R(|v|^2+\eta)^\frac{q-4}2\dm_2^2\VS\displ\hskip 4cm =-\intl\OO\n\pi\cdot vh_R(|v|^2+\eta)^{\frac{q-2}2}dx+\Sigma+J\,,\ea\ee
with  
$$\Sigma:=\intl{\po}n\cdot \n v\cdot v(|v|^2+\eta)^{q-2}\,,
$$
and 
$$J:=\frac1q\intl\OO\Delta h_R(|v|^2+\eta)^\frac q2+(v-V)\cdot \n h_R(|v|^2+\eta)^\frac q2dx\,.
$$
We now estimate the three terms on the right-hand side of (\ref{LQ-O}). 
By the H\"older inequality, we get 
$$\ba{ll}\displ\intl\OO h_R|\n\pi|| v|(|v|^2+\eta)^\frac{q-2}2dx\hskip-0.2cm&\displ\leq \intl\OO|\n \pi||v|^{q-1}dx+ c\eta\intl\OO\n\pi h_Rdx \\&\leq \dm \n \pi\dm_\frac 6{7-q}\dm v\dm_6^{q-1}+c\eta R^{\frac32}\dm \n \pi\dm_2,,\mbox{ a.e. in }t>0\,.\ea$$ 
Also, by virtue of \rf{LGPG-II} and  Sobolev embedding, we obtain
$$
\|\nabla \pi\dm_\frac 6{7-q}\dm v\dm_6^{q-1}
\leq c \dm\n\pi\dm_2\dm \n v\dm_2^{q-1}+\dm \n v\dm_2^q\dm v\dm_\frac6{4-q}\,.
$$
From \rf{GESD-II}$_1$ and, again, Sobolev embedding, we infer that $\dm v\dm_\frac6{4-q}\leq C$ for all $t>0$. Thus, recalling \rf{GESD-II}$_{2,3}$, from what we showed so far we may conclude
\be\label{LQ-II}\lim_{R\to\infty}\lim_{\eta\to0}\intll0t \big|\intl\OO h_R\n\pi\cdot v(|v|^2+\eta)^\frac{q-2}2dx\big|d\tau\leq Ct^{\frac{2-q}2}\,,\mbox{ for all }t>0.\ee
In order to estimate $\Sigma$, we begin to notice that Sobolev embedding and \rf{LGPG}$_{1,2}$ ensure that a.e. in $t>0$ $v\in C(\ov \OO)$. In particular, a.e. in $t>0$, we have $\displ\max_\po|v(t,x)|\leq\dm v(t)\dm_\infty\in L^2(0,\infty)$\,. So that, applying H\"older's inequality,  we obtain, uniformly in $\eta>0$,
\be\label{LQ-III}\intll0t\Sigma\, d\tau\leq c\dm v\dm_\infty^{q-1}\dm \n v\dm_{L^2(\po)}d\tau\leq c\intll0t\dm v\dm_\infty^{q-1}(\dm D^2v\dm_2+\dm\n v\dm_2)d\tau\leq Ct^{\frac{2-q}2}\,,\;t>0\,.\ee
Finally, we consider the term $J$. We observe that, since $h_R$ is a radial function, we have $(\omega\times x)\cdot \nabla h_R=0$   for all $x\in\OO$. Thus, $$J:=\frac1q\intl\OO\Delta h_R(|v|^2+\eta)^\frac q2+(v-\xi)\cdot \n h_R(|v|^2+\eta)^\frac q2dx\,.$$ By  the Lebesgue theorem of the dominate convergence,  we arrive at
\be
\lim_{\eta\to0}\intll0t|J|d\tau\leq cR^{-2}\intll0t\dm v\dm_{L^q(R<|x|<2R)}^qd\tau+cR^{-1}\intll0t\dm |v-\xi|^\frac1q|v|\dm_{L^q(R<|x|<2R)}^qd\tau:=J_1+J_2\,.
\ee 
Applying H\"older's inequality,  we get
\be
\lim_{R\to\infty}J_1\leq c\lim_{R\to\infty}R^{-2+3\frac{2-q}2}\intll0t\dm v\dm_{L^2(R<|x|<2R}^qd\tau=0\,.
\ee 
Concerning $J_2$, since  $|v-\xi|\leq |v|+|\xi|$, then, with obvious meaning of the symbols, we set  
$$J_2\leq J_{21}+J_{22}\,.$$ Recalling \rf{GESD-II}$_{1,6}$, we deduce
we get
\be
\displ\lim_RJ_{21}\leq c\lim_{R\to\infty}\sfrac1R\intll0t\dm v\dm_{L^{q+1}(R<|x|<2R)}^{q+1}d\tau =0\,.\label{nu}
\ee
For $J_{22}$ we distinguish the two cases: (i) $q\in (1,\frac43)$, and (ii)  $q\in[\frac43,2)$.
In case (ii), by H\"older inequality, we get
\be
\lim_{R\to\infty}J_{22}\leq c|\xi|\lim_{R\to\infty}R^{-1
+3\frac{2-q}2}\intll0t\dm  v\dm_{L^2(R<|x|<2R)}^qd\tau=0\,.
\label{nuovo}
\ee
Therefore, integrating both sides of \rf{LQ-O} from $0$ to $t$ and using \rf{LQ-II}--\rf{nuovo}, after letting first $\eta\to0$ and subsequently $R\to\infty$ we prove the lemma 
in the case $q\in[\frac43,2)$. 
If $q\in(1,\frac43)$, we  remark that, by interpolation, we get $v_0\in L^\frac43(\OO)$. Hence, by what we just proved, it follows  $\dm v(t)\dm_\frac43<\infty$ for all $t>0$. With this in mind,  applying H\"older's inequality,  we get 
$$
J_{22}\leq c|\xi|\lim_{R\to\infty}R^{-1
+3\frac{4-3q}4}\intll0t\dm  v\dm_{L^\frac43(R<|x|<2R)}^qd\tau=0\,,
$$
which, by the previous argument, leads again to \rf{LQ-I} also in case (i).
\ep
\section{Main Results}
\subsection{Initial datum $v_0\in L^q(\OO)$, $q\in(1,2)$}The objective of this section is the proof of the following theorem.
\begin{tho}\label{LDD}{\sl Any solution  of Lemma\,\ref{GESD} with $v_0$ satisfying the additional condition
$$
v_0\in L^q(\OO)\,, \ \  q\in(1,2)\,, 
$$
has the following asymptotic decay property:
\be\label{LDD-I}\dm v(t)\dm_2+|\xi(t)|+|\omega(t)|\leq \frac{C+\|v_0\|_q}{(t+1)^\frac{2-q}{4q}}\,, \mbox{ for all }t>0\,,\ee 
where, we recall,  $C>0$ depends only on $\Omega$ and the norm of the initial data entering \rf{GESD-I}.
}\end{tho}
\bp From \cite{GaLTB}, we know that the solution satisfies the energy relation:
$$\sfrac d{dt}\dm v(t)\dm_2^2+2\dm \n v(t)\dm_2^2+\sfrac d{dt}H(t)=0\,,\mbox{ for all }t>0\,,$$ where 
$H(t):=m|\xi(t)|^2+\omega(t)\cdot{\sf I}\cdot\omega(t)$\,.
We also recall the {\color{black}Gagliardo-Nirenberg inequality with non-homogeneous boundary value  \cite{CrMa}}
\be\label{GN}\dm v\dm_2\leq c\,\dm \n v\dm_2^{3\frac{2-q}{6-q}}\dm v\dm_q^\frac{2q}{6-q}\,.\ee  Setting $h(t):=\dm v(t)\dm_2^2$ and taking into account \rf{LQ-I}, from the last two displayed relations we deduce the differential inequality
$$
h'(t)+2\alpha\sfrac{h^\frac{6-q}{6-3q}}{\dm v_0\dm_q+Ct^\frac23}+H'(t)\leq 0\,,
$$
where $\alpha$ is a suitable constant.
Multiplying by $(t+1)$ both sides of differential inequality,and using Leibniz  rule, we arrive at
\be\label{LDD-II}((t+1)h(t))'+ 
2\alpha\,\sfrac{(t+1)h^\frac{6-q}{6-3q}}{\dm v_0\dm_q+Ct^\frac23}+((t+1)H(t))'\leq 
 h(t)+H(t)\,.\ee
Applying Young inequality with a suitable coefficient, we obtain for all $t>0$, 
$$
h(t)=\sfrac {(t+1)^\frac{6-3q}{6-q}h(t)}{(\dm v_0\dm_q+Ct^{\frac23})^\frac{6-3q}{6-q}}\times \sfrac {(\dm v_0\dm_q+Ct^{\frac23})^\frac{6-3q}{6-q}}{(t+1)^\frac{6-3q}{6-q}}\leq 2\alpha\sfrac{(t+1)h^\frac{6-q}{6-3q}}{\dm v_0\dm_q+Ct^\frac23}+ c\sfrac {(\dm v_0\dm_q+Ct^{\frac23})^{3\frac{2-q}{2q}}}{(t+1)^{3\frac{2-q}{2q}}}\,. $$ 
The latter allows us to increase the right-hand side of \rf{LDD-II} in the following way:
$$
((t+1)h(t))'+ 
2\alpha\sfrac{(t+1)h^\frac{6-q}{6-3q}}{\dm v_0\dm_q+Ct^\frac23}+((t+1)H(t))'\leq 
2\alpha\sfrac{(t+1)h^\frac{6-q}{6-3q}}{\dm v_0\dm_q+Ct^\frac23}+ c\sfrac {(\dm v_0\dm_q+Ct^{\frac23})^{3\frac{2-q}{2q}}}{(t+1)^{3\frac{2-q}{2q}}} +H(t)\,,
$$ 
which is equivalent to
\be\label{LDD-IIA}((t+1)h(t))'+((t+1)H(t))'\leq 
 c\sfrac {(\dm v_0\dm_q+Ct^{\frac23})^{3\frac{2-q}{2q}}}{(t+1)^{3\frac{2-q}{2q}}} +H(t)\,,
\ee
Integrating both sides of \rf{LDD-IIA}    over $[0,t]$,   and using \rf{GESD-II}$_6$, we arrive at
$$
(t+1)\Big[h(t)+H(t)\Big]\displ\leq h(0)+H(0)+ C+c\intll0t \sfrac {(\dm v_0\dm_q+C\tau^{\frac23})^{3\frac{2-q}{2q}}}{(\tau+1)^{3\frac{2-q}{2q}}}d\tau\,.
$$ 
Hence,  increasing  ${\color{black}\max\{1,t^\frac23\}}$ with $(t+1)^\frac23$, we deduce $$(t+1)\Big[h(t)+H(t)\Big] \leq h(0)+H(0)+C+c\big[\dm v_0\dm_q+C\big]\intll0t(1+\tau)^{-\frac{2-q}{2q}}d\tau\,,\mbox{ for all }t>0\,,$$ that leads to the desired result.
\ep  
\begin{rem}{\rm If we use equation  (4.13) of \cite{GaLTB} along with the estimate $$\dm \n u(t)\dm_2^2\leq c\big[\dm D^2 u(t)\dm_2^\frac23H^\frac23(t)+ H(t)+\dm \Delta u(t)\dm_2\dm u(t)\dm_2\big]\,,$$ 
from  Theorem\,\ref{LDD} it also follows the following decay 
$$
\dm \n v(t)\dm_2\leq \Big[\dm \n v_0\dm_2+ C\Big](t+1)^{-\frac {2-q}{6q}}\,,\mbox{ for all }t>0\,.$$}
\end{rem}
\subsection{Initial datum $v_0\in L^2(\OO)$}
{\color{black}In this section we   show that we still have $\|v(t)\|_2\to 0$ as $t\to\infty$, but with an unspecified rate of decay. 
\begin{tho}\label{ABLTN}{\sl There is $\delta>0$ such that any solution  given in Lemma \ref{GESD} and satisfying \rf{GESD-I} fulfills  the following  condition
\be\label{ABLN-I}\lim_{t\to\infty}\dm v(t)\dm_2=0\,.\ee}\end{tho} The proof of this theorem, given at the end of the section, needs  some preparatory result.}
\par
Let $\xi(t)$ and $\omega(t)$ be the vector functions determined in Lemma\,\ref{GESD}, pick a sequence $\{v_0^m\}\subset \mathcal V(\OO)\cap L^q(\OO)$ for some $q\in(1,2)$, and  consider the following initial-boundary value problem:
\be\label{NSm}\ba{l}\left.\ba{l}\medskip v_t^m=\n\cdot \mathbb T(v^m,\pi^m)-[(v^m-V)\cdot\n v^m+\omega\times v^m]\\
\n\cdot v^m=0
\ea\right\}
\mbox{\, in } (0,\infty)\times\OO\,,\VS v^m=(1+t)^{-\frac1m}V\mbox{\, on }(0,\infty)\times \po\,,\;\lim_{|x|\to\infty}v^m(t,x)=0\,,t\in(0,T)\,,
\\
 v^m=v_0^m(x)\,\mbox{ on }\{0\}\times\OO\,.
\ea\ee
\begin{lemma}\label{ExNSM}{\sl Let $\xi(t)$ and $\omega(t)$ be the vector functions determined in Lemma\,\ref{GESD} and, therefore, obeying \rf{goes}. Then, there is $\delta>0$ such that for all $v_0^m\in{\mathcal V}(\Omega)$ with boundary value $\overline{v_0^m}=V_0$, that satisfies
$$
\|v_0^{m}\|_{1,2}\le\delta\,,
$$
we can find a corresponding solution $(v^m,\pi^m)$ to \rf{NSm} such that , for all $T>0$ and $R\ge R_*$,
\be\label{ExNSM-I}\ba{l} v^m\in L^\infty(0,\infty;W^{1,2}(\OO))\cap L^2(0,\infty;W^{2,2}(\OO))\,, \;\n\pi^m\in L^2(0,\infty;L^{2}(\OO))\,,\VS v^m\in C([0,T];W^{1,2}(\OO_R))\,,\;v_t^m\in L^2(0,T;L^{2}(\OO_R))\,.\ea\ee 
Moreover, assuming $v_0^m\in L^q(\OO)$ for some $q\in (1,2)$,  we have also
\be\label{LQ_0} \dm v^m(t)\dm_q\leq c(v_0^m)(1+t)^\frac {2-q}{2q}\,,\mbox{ for all }t>0\,.\ee
}\end{lemma}
\begin{proof} The proof of existence of solutions to \rf{NSm} in the class \rf{ExNSM-I} under  smallness assumption of the data can be viewed as a particular (actually simpler) case of
that employed for Lemma \ref{GESD}. We will just sketch the main aspects, referring the reader to \cite{GaLTB} to fill the missing steps. In the first place, we lift the velocity $V$ at the boundary. Precisely,  there exists  a solenoidal function $\widehat{V}\in W^{1,2}(0,\infty;W^{2,2}(\Omega))$ of bounded support in $\Omega$  such that
\be\ba{ll}\medskip
\widehat{V}(x,t)=V(x,t)\,,\ (t,x)\in (0,\infty)\times\Omega\,,\\ \medskip
\|\widehat{V}\|_{W^{1,2}(0,\infty;W^{2,2}(\Omega))}\le c\, \left(|\xi|_{W^{1,2}(0,\infty)}+|\omega|_{W^{1,2}(0,\infty)}\right)\,,\\
|(u\cdot\n\widehat{V},u)|\le \mbox{$\frac14$}\|\n u\|_2^2\,,\ \ u\in W^{1,2}(\Omega)\,;
\ea\label{V}
\ee
see \cite[Lemma 2.2]{GaSi}.
Thus, setting $u^m:=v^m-h^m\widehat{V}$, $h^m:=(1+t)^{-\frac1m}$, and omitting all superscripts for simplicity, from \rf{NSm} we deduce that $u\equiv u^m$ satisfies the following problem
\be\label{NSm1}\ba{l}\left.\ba{l}\medskip u_t=\n\cdot \mathbb T(u,\pi)-(u-V)\cdot\n u+\omega\times u-hu\cdot\n \widehat{V}-h\widehat{V}\cdot\n u+f\\
\n\cdot u=0
\ea\right\}
\mbox{\, in } (0,\infty)\times\OO\,,\VS u=0\mbox{\, on }(0,\infty)\times \po\,,
\\
 u=v_0(x)-\widehat{V}(x,0)\,\mbox{ on }\{0\}\times\OO\,.
\ea\ee
where $f\in L^2(0,\infty;L^{2}(\Omega))$ has bounded support in $\Omega$ and is such that
\be\label{f}
\int_0^\infty\|f(t)\|_2^2dt\le c\,\int_0^\infty(|\xi(t)|^2+|\omega(t)|^2+|\dot{\xi}(t)|^2+|\dot{\omega}(t)|^2)dt
. 
\ee
It is obvious that if we show existence of a solution $(u,\pi)$ to \rf{NSm1} in the class \rf{ExNSM-I}, then the corresponding $(v,\pi)$ is a solution to the original problem in the same class. 
Now, the existence proof to \rf{NSm1} is  based on the so called ``invading domains" technique. Namely, let $\{\Omega_{R_k}\}$ be an increasing sequence of (bounded) domains with $\cup_{k=1}^\infty\Omega_{R_k}=\Omega$. In each $\Omega_{R_k}$ we consider the ``approximating problem" 
\be\label{NSmk}\ba{l}\left.\ba{l}\medskip u_t=\n\cdot \mathbb T(u,\pi)-(u-V)\cdot\n u+\omega\times u-hu\cdot\n \widehat{V}-h\widehat{V}\cdot\n u+f\\
\n\cdot u=0
\ea\right\}
\mbox{\, in } (0,\infty)\times\OO_{R_k}\,,\VS u=0\mbox{\, on }(0,\infty)\times \po\cup \partial B_{R_k}\,,
\\
 u=v_0(x)-\widehat{V}(x,0)\,\mbox{ on }\{0\}\times\OO_{R_k}\,.
\ea\ee
The crucial point of the above technique is to show existence of solutions to \rf{NSmk} in the class \rf{ExNSM-I}, with bounds for the corresponding norms that are independent of $k$. All these properties, in our case, are shown by using the Galerkin method, with a special basis constituted by eigenvectors of the Stokes operator \cite[Section 4]{GaSi}. The basic estimates are then obtained by formally testing \rf{NSmk}$_1$ by $u$ on time and, a second time by ${\sf P}\Delta u$, where ${\sf P}$ is the Helmholtz-Weyl projector. We thus obtain the following two ``energy relations"
\be\label{energy}\ba{ll}\medskip
\sfrac12\sfrac d{dt}\|u(t)\|_2^2+\|\nabla u\|_2^2= -h(u\cdot\n\widehat{V},u)+(f,u)\\ \medskip
\sfrac12\sfrac d{dt}\|\nabla u(t)\|_2^2+\|{\sf P}\Delta u|_2^2=-((u-\xi)\cdot\n u,{\sf P}\Delta u) +(\omega\times x\cdot\n u-\omega\times u,{\sf P}\Delta u)\\ 
\hspace*{4.4cm}-(hu\cdot\n \widehat{V}-h\widehat{V}\cdot\n u+f,{\sf P}\Delta u)\,,
\ea
\ee
where $(\cdot,\cdot)$ denotes the $L^2(\Omega_{R_k})$-scalar product. Taking into account \rf{V}$_{2,3}$, from \rf{energy}$_1$ and Sobolev inequality $\|u\|_6\le c\,\|\n u\|_2$, we easily deduce
\be\label{en1}
\sfrac d{dt}\|u(t)\|_2^2+\|\n u\|_2^2\le c\,\|f \|_2^2
\ee
with $c$ independent of $k$. 
Furthermore, arguing as in \cite[Theorem 2.1]{GaLTB}, we show
\be\ba{rl}\medskip
|((u-\xi)\cdot\n u,{\sf P}\Delta u)|
&\!\!\!\le c_\eta\,(\|\n u\|_2^2+\|\n u\|_2^6)+\eta\|D^2 u\|_2^2\\
|(\omega\times x\cdot\n u-\omega\times u,{\sf P}\Delta u)|&\!\!\!\le c_\eta\,\|\n u\|_2^2+\eta \|D^2 u\|_2^2
\ea
\label{en2}
\ee
where $\eta>0$ is arbitrary and $c_\eta$ does not depend on $k$. Also, again by  \rf{V},  and Sobolev inequality we infer
\be\label{en3}
|(hu\cdot\n \widehat{V}-h\widehat{V}\cdot\n u+f,{\sf P}\Delta u)|\le c_\eta\,(\|\n u\|_2^2+\|f\|_2^2)+\eta\|D^2 u\|_2^2\,.
\ee
We next recall the following inequality, due to J.G. Heywood \cite[Lemma 1]{Hey}
\be\label{Hey}
\|D^2u\|_{2,\Omega_{R_k}}\le c\,(\|{\sf P}\Delta u\|_{2,\Omega_{R_k}}+\|\n u\|_{2,\Omega_{R_k}})\,,
\ee
where $c$ is independent of $k$.
Thus, employing \rf{en2}, \rf{en3} and \rf{Hey} in \rf{energy}$_2$, and taking $\eta$ sufficiently small, we deduce
\be\label{MC}
\sfrac d{dt}\|\nabla u(t)\|_2^2+c_1\|D^2 u\|_2^2\le c_2\,(\|\n u\|_2^2+\|\n u\|_2^6+\|f\|_2^2)\,,
\ee
where both $c_1$ and $c_2$ are independent of $k$. Next, integrating both sides of  \rf{en1},  using \rf{V}$_2$, \rf{f} and classical embedding theorems, we obtain
\be\ba{rl}\medskip\label{ok}
\essup{t>0}\,\|u(t)\|_2^2+\displaystyle \int_0^\infty\|\nabla u(t)\|_2^2dt&\!\!\!\le \Big(\|v_0^m\|_2^2+\|\widehat{V}(0)\|_2^2+c\displaystyle\int_0^\infty\|f(t)\|_2^2dt\Big)\\
&\!\!\!\le c\left(\|v_0^m\|_2^2+\|\xi\|_{W^{1,2}(0,\infty)}^2+\|\omega\|_{W^{1,2}(0,\infty)}^2\right)\,.
\ea
\ee
As a result, from \rf{MC} and \cite[Lemma 4.3]{GaLTB} it follows that if the right-hand side of \rf{ok} is taken less than a sufficiently small $\delta$ (the latter is ensured by the assumption on $v_0^m$ and \rf{goes}), there is $C_1>0$, independent of $k$, such that
\be\label{D}
\essup{t>0}\,\|\n u(t)\|_2\le C_1\delta\,.
\ee
Replacing this information back into \rf{MC},  integrating both sides of the resulting inequality over $(0,\infty)$ and keeping in mind \rf{ok}, we also deduce
\be
\label{D2}
\int_0^\infty\|D^2 u(t)\|_2^2dt\le C_2\,\delta\,,
\ee
where $C_2$ does not depend on $k$.
Collecting \rf{ok}--\rf{D2}, we infer that
\be\label{EsT}
u\in L^\infty(0,\infty;W^{1,2}(\Omega))\cap L^2(0,\infty;W^{2,2}(\Omega))
\ee
with corresponding norms uniformly bounded in $k$. Furthermore, let $R_{k_0}$ be arbitrarily fixed. Then, proceeding exactly as in \cite[Lemma 4.3]{GaSi} we can show the following estimate, for all $T>0$, and all $k\ge k_0$
\be\label{nero}
\int_0^T\|u_t(t)\|_{2,\Omega_{R_{k_0}}}^2dt\le C_0\,,\ \ \int_0^T\|\n\pi(t)\|_2^2dt\le C
\ee
where $C_0$ depends on $k_0$ but not on $k$, whereas $C$ only on the data and $\Omega$. Combining \rf{EsT} and \rf{nero} with Galerkin method, we thus prove the existence of a solution to  \rf{NSmk} in the class \rf{ExNSM-I}, with $\Omega$ and $\Omega_R$ replaced by $\Omega_{R_k}$; see \cite{GaSi} for details. However, since all the previous bounds are independent of $k$, we can pass to the limit $k\to\infty$ and show that the sequences $\{(u_k,\pi_k)\}$ of solutions to \rf{NSmk} converge, in suitable topologies to a solution $(u,\pi)$ of the original problem \rf{NSm} that belongs in the class \rf{ExNSM-I}; see \cite{GaSi} for details. Next, employing the same arguments of \cite[Lemma 3.2 and Proposition 2.1]{GaLTB} we can show that the solution $(v^m,\pi^m)$ just obtained satisfies also the properties stated in Lemmas \ref{LGPG} and \ref{GPSMD} and, by exactly the same proof, Lemma \ref{ASS}. As a result, we can follow the same steps of the proof of Proposition \ref{LQ_1} to prove \rf{LQ_0}.
\end{proof}
\begin{lemma}\label{VMAB}{\sl Let $(v^m,\pi^m)$ be the solution given Lemma\,\ref{ExNSM}. Then, 
\be\label{VMAB-I}\dm v^m(t)\dm_2\leq c(\dm v_0^m\dm_2,\dm v_0^m\dm_q)(1+t)^{-\frac1m}\,,\mbox{ for all }t>0\,.\ee}\end{lemma}
\bp Testing \rf{NSm}$_1$ with $v^m$, we get
\be\label{enre}\sfrac12\sfrac d{dt}\dm v^m(t)\dm_2^2+\dm \n v^m\dm_2^2=(1+t)^{-\frac1m}\intl{\po}n\cdot \mathbb T(v^m,\pi^m) \cdot V\,.\ee
By the trace theorem, it follows
$$\Big|\intl{\po}n\cdot \mathbb T(v^m,\pi^m) \cdot V\Big|\leq c\Big[\dm \n\cdot \mathbb T(v^m,\pi^m)\dm_2+\dm \mathbb T(v^m,\pi)\dm_2\Big]\big[|\xi(t)|+|\omega(t)|\big]\,,\mbox{ for all }t>0\,.$$
Moreover, {\color{black} employing again \rf{GN}}, we find
$$\dm v^m\dm_2\leq c\dm \n v^m\dm_2^a\dm v^m\dm_q^{1-a}\,,\mbox{ with }a:=\sfrac{3(2-q)}{6-q}\,.$$
Thus, combining the latter with \rf{LQ_0} and \rf{enre}, we show
\be\label{VMAB-III}\sfrac d{dt}\dm v^m\dm_2^2+c\dm v^m\dm_2^{2\frac{6-q}{3(2-q)}}\dm v_0\dm_q^{-\frac{4q}{3(2-q)}}(1+t)^{-\frac23}\leq G^m(t)(1+t)^{-\frac1m}\,,\mbox{ for all }t>0\,,\ee where we set $$G^m(t):= c\Big[\dm \n\cdot \mathbb T(v^m,\pi^m)\dm_2+\dm \mathbb T(v^m,\pi)\dm_2\Big]\big[|\xi(t)|+|\omega(t)|\big] \,.$$ We  have \be\label{AUX}G^m(t)\in L^1(0,\infty)\,,\mbox{ with a bound uniform with respect to } m\,.\ee This follows from  Lemma \,\ref{ExNSM} and \rf{goes}.
If we multiply both sides of  \rf{VMAB-III} by $(1+t)$ and apply the rule for derivative of a product, we get
\be\label{VMAB-II}\sfrac d{dt}(1+t)\dm v^m\dm_2^2+c_0\dm v^m\dm_2^{2\frac{6-q}{3(2-q)}}\dm v_0\dm_q^{-\frac{4q}{3(2-q)}}(1+t)^{\frac13}\leq \dm v^m\dm_2^2+ G^m(t)(1+t)^{1 -\frac1m}\,,\mbox{ for all }t>0\,.\ee
Using in the latter the Young inequality, we get
$$\dm v^m\dm_2^2=\sfrac{\dm v^m\dm_2^2(1+t)^\frac{2-q}{6-q}}{\dm v_0^m\dm_q^\frac{4q}{6-q}}\times\sfrac{\dm v_0^m\dm_q^ \frac{4q}{6-q}}{(1+t) ^\frac{2-q}{6-q}}\leq c_0\dm v^m\dm_2^{2\frac{6-q}{3(2-q)}}\dm v_0^m\dm_q^{-\frac{ 4q}{3(2-q)}}(1+t)^{\frac13}+c\dm v_0^m\dm_q^2(1+t)^{-\frac{2-q}{2q} }\,.$$ Thus, if we  employ this inequality on the right hand side of \rf{VMAB-II},  integrate over $(0,t)$, and bear in mind \rf{AUX}, we arrive  to show \rf{VMAB-I}. 
\ep
The next result proves that the sequence $\{v^m\}$ converges in $L^2$ to $v$,  uniformly in $t>0$, provided only $v_0^m\to v_0$ in $L^2$.
\begin{lemma}\label{LDC}{\sl Let $v_0\in \mathcal V(\OO)$ and assume $\{v_0^m\}\subset \mathcal V(\OO)$ converges to $v_0$ in $L^2(\OO)$. Then the sequence of solutions $\{(v^m,\pi^m)\}$ corresponding to $\{v_0^m\}$ converges in $L^2$, uniformly in time, to the solution $(v,\pi)$  to \rf{NS} given in Lemma \ref{GESD}.}\end{lemma}
\bp We set $w^m:=v^m-v$ and $p^m:=\pi^m-\pi$. Then the pair $(w^m,p^m)$ is a solution to the problem
\be\label{LDC-I}\ba{l}w_t^m-\n\cdot \mathbb T(w^m,p^m)=-(w^m-V)\cdot\n w^m -w^m\cdot \n v-v\cdot\n w^m+\omega\times w^m\,, \VS \n\cdot w^m\,,\mbox{ in }(0,T)\times\OO\,,\VS w^m=\big[1-(1+t)^{-\frac1m}\big]V(t, x)\mbox{ on }(0,T)\times\po\,,\;w^m=v^m_0-v_0\mbox{ on }\{0\}\times\OO\,.\ea\ee
Testing \rf{LDC-I}$_1$ by $w^m$ and integrating by parts on $(0,T)\times\OO$, we get 
\be\label{MiNa}\sfrac 12\sfrac d{dt}\dm w^m\dm_2^2+\dm \n w^m\dm_2^2=\big[1-(1+t)^{-\frac1m}\big]
\intl{\po}n\cdot \mathbb T(w^m,p^m)\cdot V+(w^m\cdot\n w^m,v)\,.
\ee
Applying   H\"older inequality  and   Sobolev inequality with non-homogeneous  boundary value  for $w^m$ \cite{GaMon}, we deduce 
$$\big|(w^m\cdot\n w^m,v)\big|\leq c\dm \n w^m\dm_2^2\dm v\dm_3\,.$$
Thanks to the interpolation inequality with non-homogeneous boundary value \cite{CrMa} 
$$
\|v\|_3\le c\,\|v\|_2^{\frac12}\|\nabla v\|_2^{\frac12}
$$
and \rf{goes}, it follows
\be\label{LDC-II}
\big|(w^m\cdot\n w^m,v)\big|\leq c\,\delta\,\|\nabla w^m\|_2^2.
\ee
Moreover, by standard trace theorems and Schwarz inequality  we get
$$H^m(t):=\Big|\intl{\po}n\cdot \mathbb T(w^m,p^m)\cdot V\Big|\leq c\big[\dm \n\cdot \mathbb T(w^m,\pi^m)\dm_2+\dm \mathbb T(w^m,p^m)\dm_2\big]\big[|\xi|+|\omega|\big]\,.$$
In view of Lemmas \ref{GPSMD} and \ref{ASS}, and of the statements made at the end of the proof of Lemma \ref{ExNSM}, we infer
$$\label{LDC-III}
\int_0^\infty \big[\dm \n\cdot \mathbb T(w^m,\pi^m)\dm_2+\dm \mathbb T(w^m,p^m)\dm_2\big]^2\le C
$$ 
where $C$ is independent of $m$.
Thus,
$$
\int_0^\infty[1-(1+t)^{\frac1m}]H^m(t)dt\le c\left(\int_0^\infty[1-(1+t)^{\frac1m}]^2[|\xi(t)|+|\omega(t)|]^2dt\right)^\frac12
$$
Consequently, using the latter along with \rf{LDC-II} in \rf{MiNa}, and taking $\delta$ sufficiently small, we deduce 
\be\label{LDC-IV}\dm w^m(t)\dm_2^2\leq \dm w^m(0)\dm_2^2+c\intll0\infty\big[1-(1+t)^{-\frac1m}\big]^2\big[|\xi(t)|+|\omega(t)|\big]^2dt\,,\ee
with $c$ independent of $m$.
Letting $m\to \infty$ in \rf{LDC-IV}, and observing that the sequence $\{1-(1+t)^{-\frac1m}\}$  converges to zero, the property \rf{goes} along with the dominated convergence theorem allows us to conclude the proof.\ep
 \bp[Proof of Theorem\,\ref{ABLTN}] By the previous lemma, for a given $\varepsilon>0$ we can find $\ov{m}\in\mathbb N$ depending only on $\varepsilon$ such that
$$
\|v^{\ov{m}}(t)-v(t)\|_2<\varepsilon\,,
$$
and therefore, by the triangle inequality,
$$
\|v(t)\|_2<\varepsilon +\|v^{\ov{m}}(t)\|_2\,.
$$
As a consequence, using
 the results of Lemma\,\ref{VMAB} we get 
$$\limsup_{t\to\infty}\dm v(t)\dm_2<\varepsilon\,,
$$ 
which, by the arbitrariness of $\varepsilon$, proves the theorem. 
\ep

\medskip\par\noindent
{\bf Acknowledgment.} The work of G.P.~Galdi is partially supported by NSF grant DMS-2307811. The work of P. Maremonti is supported by GNFM (INdAM).

\end{document}